\documentclass[a4paper,12pt]{article}

\usepackage{lineno,hyperref}
\usepackage{mathrsfs}
\usepackage{bm}
\usepackage{float}
\usepackage{amssymb,amsmath,amsthm,latexsym}
\usepackage{amsfonts}
\usepackage{graphicx}
\usepackage{pdflscape}
\usepackage{setspace, subcaption}
\usepackage{mathtools}
\DeclarePairedDelimiter\ceil{\big\lceil}{\big\rceil}
\DeclarePairedDelimiter\floor{\big\lfloor}{\big\rfloor}
\usepackage{pgfplots}
\pgfplotsset{compat=1.15}
\usepackage{mathrsfs}
\usepackage{rotating}

\usetikzlibrary{arrows}
\usetikzlibrary{shapes}
\newtheorem{theorem}{Theorem}[section]
\newtheorem{coro}{Corollary}[theorem]

\newtheorem{definition}[theorem]{Definition}
\newtheorem{obs}[theorem]{Observation}

\begin{document}
	
	\begin{center}
		{\large \bf {Some new generalizations of Domination using restrictions on degrees of vertices}}\\
		{\large\vspace{0.1in} 
			Shyam S. Kamath\footnote{Corresponding author} and Nithya Muraleedharan$^2$}\\
		\vspace{0.1in}
		{\small \it $^{1,2}$ Department of Mathematical and Computational Sciences \\
			National Institute of Technology Karnataka, Surathkal, \\
			Srinivasnagar, Mangalore - 575 025, India}\\
		{\small \it E-mails: $ ^1 $shyam@nitk.edu.in, 
			$ ^2 $nithyamtr@gmail.com}
	\end{center}
	
	\begin{abstract}
		\noindent A set $D$ of vertices in a graph $G=(V,E)$ is a \textit{degree restricted dominating set} for $G$ if each vertex $v_i$ in $D$ is dominating atmost $g(d_i)$ vertices of $V-D$, where $g$ is a function restricting the degree value $d_i$ with respect to the given function value $k_i$ for a natural valued function $f$ from the vertex set of the graph. We define three different types of Degree Restricted Domination by varying the way how the restricted function $g(v_i)$ is defined. If $g(d_i)=\ceil {\frac{d_i}{k_i}}$, the corresponding domination is called the \textit{ceil degree restricted domination}, in short, $CDRD$, and the dominating set obtained in this manner is the $CDRD$-set. If $g(d_i)=\floor{\frac{d_i}{k_i}}$ or $g(d_i)=d_i-k_i+1$, then the corresponding dominations are respectively called the \textit{floor degree restricted domination}, in short $FDRD$, or the \textit{translate degree restricted domination}, $TDRD$. The dominating sets obtained in this manner are the $FDRD$-set and the $TDRD$-set respectively. In this paper, we introduce these new generalizations of the domination number in line with the different $DRD$-sets and study these types of domination for some classes of graphs like complete graphs, caterpillar graphs etc. Degree restricted domination has a vital role in retaining the efficiency of nodes in a network and has many interesting applications. \\ 
		
		\vspace{2mm}
		
		
		\vspace{2mm}
		
		
	\end{abstract}


	\section{Introduction}

		Let $G=(V,E)$ be a graph with order $n$ and size $m$, where $V=\{v_1,v_2, \hdots, v_n\}$. The \textit{degree} of a vertex $v_i \in V$ is the number of edges incident with it and is denoted by $d_G(v_i)$ or $d_i$.	A \textit{caterpillar} graph is a tree which can be obtained from a path by adding pendant edges with its vertices. The initial path sans the pendant vertices is called the \textit{spine} of the caterpillar.\\
		Two vertices $v_i$ and $v_j$ \textit{dominate} each other in a graph $G=(V,E)$ if $v_i$ and $v_j$ are adjacent in $G$, i.e., $v_iv_j \in E$. A set $D \subseteq V$ in a graph $G$ is called a \textit{dominating set} if every vertex in $V-D$ is dominated by atleast one vertex in $D$. Property of domination is superhereditary and so, the minimal dominating sets are of much importance. The minimum cardinality of a minimal dominating set is called the \textit{domination number}, denoted as $\gamma(G)$.\\
		In networks, to retain the efficiency of those nodes which are in contact with more number of nodes, we may have to restrict the transfer of data only through a certain pairs of nodes. This restriction can be done in various forms. If the number of such data transformation is restricted equally at every node, that is, if every vertex $v_i$ can dominate atmost $\ceil{\frac{d_i}{k}}$ vertices adjacent to it, then such a domination is called \textit{$k$-part degree restricted domination} \cite{kamath20162drd,kamath2019relation}. But practically such restriction need not be uniform. It can vary depending on the situation on the type of the network and its applications. In this paper, we model some such restrictions through graphs.\\
		The reader is referred to \cite{west2001introduction} for the notations and terminologies and \cite{haynes1998fundamentals,haynes1998domination} for the domination concepts. 
	
	\section{Main Results} \label{sec1}
	
	\subsection{Degree Restricted Domination}
		\begin{definition}
			Let $G=(V,E)$ be a graph with vertex set $V=\{v_1,v_2,\hdots,v_n\}$ and let the degree sequence be $(d_1,d_2,\hdots,d_n)$ where $d_i=d(v_i)$. Suppose $f:V \rightarrow \mathbb{N}$ is a function defined as $f(v_i)=k_i$, where $1 \leq k_i \leq d_i$ and $f(v_i)=1$ if $v_i$ is an isolated vertex. A dominating set $D \subseteq V$ is a degree restricted dominating set for the graph $G$ if each vertex $v_i$ in $D$ is dominating atmost $g(d_i)$ vertices of $V-D$, where $g$ is a function restricting the degree value $d_i$ with respect to the given function $f$.
		\end{definition}
		\noindent
		By varying the way the function $g$ is defined we get different generalizations for the dominating sets. We define here three types of degree restricted domination.
	\subsubsection{Ceil Degree Restricted Domination(CDRD)}
	If $g(d_i)=\ceil{\frac{d_i}{k_i}}$, the corresponding domination is called the \textit{ceil degree restricted domination}, in short $CDRD$, and a dominating set obtained in this manner is a $CDRD$-set. The minimum cardinality of a $CDRD$-set, the \textit{$CDRD$ number} of $G$ is denoted as $\gamma_{\overline{f}}(G)$ or $\gamma_{\overline{f}}$. A $CDRD$-set with minimum cardinality is a $\gamma_{\overline{f}}$-set.
	\begin{obs}
		\begin{enumerate}
			\item If $k_i=1$, for all $v_i \in V$, then the $CDRD$ is same as the fundamental domination. Thus $\gamma_{\overline{f}}(G)=\gamma(G)$ in this case.
			\item If $k_i=k$, where $k \leq \delta(G)$ for each $v_i \in V$, then the corresponding domination is the $k$-part degree restricted domination defined in \cite{kamath2019relation}.
		\end{enumerate}
	\end{obs}
	
	\begin{figure}[h]
		\centering
		\begin{tikzpicture}[line cap=round,line join=round,>=triangle 45,x=0.75cm,y=0.75cm]
			\clip(5.,-3) rectangle (15.,2.7);
			\draw [line width=1pt] (6.,0.)-- (8.,2.);
			\draw [line width=1pt] (8.,2.)-- (10.,0.);
			\draw [line width=1pt] (10.,0.)-- (12.,2.);
			\draw [line width=1pt] (10.,0.)-- (14.,0.);
			\draw [line width=1pt] (12.,-2.)-- (10.,0.);
			\draw [line width=1pt] (10.,0.)-- (8.,-2.);
			\draw [line width=1pt] (6.,0.)-- (8.,-2.);
			\draw [line width=1pt] (12.,2.)-- (14.,0.);
			\begin{scriptsize}
				\draw [fill=black] (8.,2.) circle (2.5pt);
				\draw[color=black] (8.,2.4) node {{\normalsize $v_1(2)$}};
				\draw [fill=black] (8.,-2.) circle (2.5pt);
				\draw[color=black] (8.,-2.4) node {{\normalsize $v_6(1)$}};
				\draw [fill=black] (6.,0.) circle (2.5pt);
				\draw[color=black] (5.8,0.5) node {{\normalsize $v_3(1)$}};
				\draw [fill=black] (10.,0.) circle (2.5pt);
				\draw[color=black] (10.,0.7) node {{\normalsize $v_4(3)$}};
				\draw [fill=black] (12.,2.) circle (2.5pt);
				\draw[color=black] (12.,2.4) node {{\normalsize $v_2(2)$}};
				\draw [fill=black] (12.,-2.) circle (2.5pt);
				\draw[color=black] (12.,-2.4) node {{\normalsize $v_7(1)$}};
				\draw [fill=black] (14.,0.) circle (2.5pt);
				\draw[color=black] (14.3,0.5) node {{\normalsize $v_5(2)$}};
			\end{scriptsize}
		\end{tikzpicture}
		\caption[$G_1$]{Graph $G$}
		\label{fig:graph-1}
	\end{figure}
	\noindent
	For the graph $G$ in Fig.\ref{fig:graph-1} with the given function $f(v_1)=2, \, f(v_2)=2,\, f(v_3)=1, \, f(v_4)=3, \, f(v_5)=2, \, f(v_6)=1$ and $f(v_7)=1$, indicated in the parantheses, $v_1,v_2,v_5,v_7$ can dominate atmost one vertex and $v_3,v_4,v_6$ can dominate atmost two vertices in accordance with $CDRD$. Thus $\{v_2,v_3,v_4\}$ forms a minimal $CDRD$ set and which is also a minimum $CDRD$-set. Thus $\gamma_{\overline{f}}(G)=3$.
	
	\subsubsection{Floor Degree Restricted Domination(FDRD)}
	If $g(d_i)=\floor{ \frac{d_i}{k_i} }$, the corresponding domination is the \textit{floor degree restricted domination}, in short $FDRD$, and a dominating set obtained in this manner is a \textit{$FDRD$-set}. The minimum cardinality of a $FDRD$-set, \textit{$FDRD$ number} is denoted as $\gamma_{\underline{f}}(G)$ or $\gamma_{\underline{f}}$. A $FDRD$-set with minimum cardinality is a $\gamma_{\underline{f}}$-set.
	\begin{obs}
		When each $d_i$ is divisible by the corresponding $k_i$, for $i=1,2,\hdots,n$ then $\ceil{ \frac{d_i}{k_i} }=\floor{ \frac{d_i}{k_i}}$ and hence the $CDRD$-set and $FDRD$-set will be same.
	\end{obs}
	
	\noindent
	In Fig.\ref{fig:graph-1} $v_1,v_2,v_4,v_5,v_7$ can dominate atmost one vertex and $v_3,v_6$ can dominate atmost two vertices in accordance with $FDRD$. Thus $\{v_2,v_3,v_4\}$ forms a minimal $FDRD$ set and which is also a minimum $FDRD$-set. Thus $\gamma_{\underline{f}}(G)=3$.
	
	\subsubsection{Translate Degree Restricted Domination(TDRD)}
	If $g(d_i)=d_i-k_i+1$, then such a domination is the \textit{translate degree restricted domination}, in short $TDRD$, and such dominating set is called a \textit{$TDRD$-set}. The minimum cardinality of a $TDRD$-set, \textit{$TDRD$ number} is denoted as $\gamma_{f_t}(G)$ or $\gamma_{f_t}$. A $TDRD$-set with minimum cardinality is a $\gamma_{f_t}$-set.\\
	\noindent
	In Fig.\ref{fig:graph-1} $v_1,v_2,v_5,v_7$ can dominate atmost one vertex, $v_3,v_6$ can dominate atmost two vertices and $v_4$ can dominate atmost 3 vertices in accordance with $TDRD$. Thus $\{v_2,v_3,v_7\}$ forms a minimal $TDRD$ set; but is not a minimum $TDRD$-set. Here $\{v_3,v_4\}$ forms a minimum $TDRD$-set and thus $\gamma_{f_t}(G)=2$.
	\begin{obs}
		The newly defined domination varies for the same graph with different function values. Consider the graph $G$ in Fig.\ref{fig:graph-1} with different function value as $f(v_1)=2, \, f(v_2)=2,\, f(v_3)=1, \, f(v_4)=2, \, f(v_5)=2, \, f(v_6)=1$ and $f(v_7)=1$, then $v_4$ can dominate atmost three vertices with respect to $CDRD$ and thus $\gamma_{\overline{f}}(G)=2$ with the $\gamma_{\overline{f}}$-set $\{v_3,v_4\}$.
	\end{obs}
	
	\subsection{\textit{DRD} Number}
	In a graph $G$ with the vertex set $\{v_1,v_2,\hdots,v_n\}$, a vertex can dominate maximum number of vertices if $k_i=1$ for every $i$. Then as observed above $\gamma_{\overline{f}}(G)=\gamma(G)$. A vertex $v_i$ can dominate only one of its neighbours when $k_i=d_i$, and thus the $CDRD$ number will be maximum if all the vertices have $k_i=d_i$. In a star graph $K_{1,n-1}$, if $k_i=d_i$ for each vertex, then the central vertex will dominate one of its neighbours and all other vertices must be in the $CDRD$ set. Thus $\gamma_{\overline{f}}(K_{1,n-1})=n-1$. Hence $\gamma \leq \gamma_{\overline{f}} \leq n-1$.
	\begin{theorem}
	\label{thm}
		For any graph $G$, $\gamma(G) \leq \gamma_{\overline{f}}(G) \leq \gamma_{\underline{f}}(G)$.
	\end{theorem}
	\begin{proof}
		Any $CDRD$-set or $FDRD$-set is also a dominating set for any graph $G$. Thus $\gamma(G) \leq \gamma_{\overline{f}}(G)$ and $\gamma(G) \leq \gamma_{\underline{f}}(G)$. Also, for any vertex $v_i$ (with the function value $k_i$ and the degree $d_i$), $\floor{ \frac{d_i}{k_i}} \leq \ceil{ \frac{d_i}{k_i} }$ and hence $\gamma_{\overline{f}}(G) \leq \gamma_{\underline{f}}(G)$.\\
		Thus $\gamma(G) \leq \gamma_{\overline{f}}(G) \leq \gamma_{\underline{f}}(G)$.
	\end{proof}
	\begin{coro}
	\label{cor1}
		If $G$ is a graph for which degree $d_i$ of each vertex is divisible by the corresponding function value $k_i$, then $\gamma_{\overline{f}}(G)=\gamma_{\underline{f}}(G)$.
	\end{coro}
	\begin{proof}
		Since every $k_i$ divides $d_i$, we have $\ceil{ \frac{d_i}{k_i} } = \floor{ \frac{d_i}{k_i} }$. So, every CDRD-set must be a FDRD-set and hence the result.
	\end{proof}
	\begin{coro}
	\label{cor2}
		If $G$ is a graph for which degree $d_i$ of each vertex in a $\gamma_{\overline{f}}$ or $\gamma_{\underline{f}}$-set is divisible by the corresponding function value $k_i$, then $\gamma_{\overline{f}}(G)=\gamma_{\underline{f}}(G)$.
	\end{coro}
	\begin{coro}
	\label{cor3}
		If $G$ is a graph for which each vertex has the function value $k_i=1$, then $\gamma(G) = \gamma_{\overline{f}} (G) = \gamma_{\underline{f}} (G)$.
	\end{coro}
	\noindent
	The corollaries to Theorem \ref{thm} are only the sufficient conditions. Figure \ref{fig:graph-3} is a counter example for the converse of above stated corollaries where the function values for the vertices are indicated in the parantheses.
	\begin{figure}[h]
		\centering
		\begin{tikzpicture}[line cap=round,line join=round,>=triangle 45,x=1.0cm,y=1.0cm]
			\clip(1.,1.) rectangle (8.,5.);
			\draw [line width=1.pt] (2.,4.)-- (5.,4.);
			\draw [line width=1.pt] (5.,4.)-- (5.,2.);
			\draw [line width=1.pt] (5.,2.)-- (2.,2.);
			\draw [line width=1.pt] (2.,2.)-- (2.,4.);
			\draw [line width=1.pt] (5.,2.)-- (7.,2.);
			\begin{scriptsize}
				\draw [fill=black] (2.,4.) circle (2.5pt);
				\draw[color=black] (2.,4.3) node {{\normalsize $v_1(1)$}};
				\draw [fill=black] (2.,2.) circle (2.5pt);
				\draw[color=black] (2.,1.7) node {{\normalsize $v_3(2)$}};
				\draw [fill=black] (5.,2.) circle (2.5pt);
				\draw[color=black] (5.,1.7) node {{\normalsize $v_4(2)$}};
				\draw [fill=black] (5.,4.) circle (2.5pt);
				\draw[color=black] (5,4.3) node {{\normalsize $v_2(2)$}};
				\draw [fill=black] (7.,2.) circle (2.5pt);
				\draw[color=black] (7.,1.7) node {{\normalsize $v_5(1)$}};
			\end{scriptsize}
		\end{tikzpicture}
		\caption{Counter example for the converse of Corollaries to Theorem \ref{thm}}
		\label{fig:graph-3}
	\end{figure}
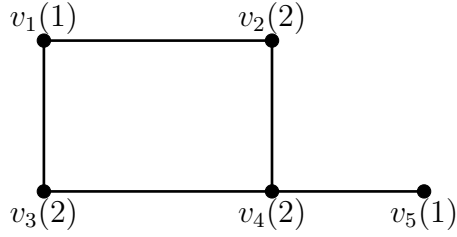
	\begin{theorem}
		For any graph $G$, $\bigg \lceil \frac{n}{1+\ceil{\frac{\Delta}{k}}}\bigg \rceil \, \leq \, \gamma_{\overline{f}} \, \leq \, n-\underset{i \in [n]} {\mbox{max }}\bigg\lceil\frac{d_i}{k_i}\bigg\rceil$ where $k\, =\underset{i \in [n]}{\mbox{min }}k_i$.
	\end{theorem}
	\begin{proof}
		For any vertex $v_i$ in a $\gamma_{\overline{f}}$-set, it can dominate atmost $1+\ceil{\frac{d_i}{k_i}}$ vertices. Maximum possible value for $\ceil{\frac{d_i}{k_i}}$ is when $d_i=\Delta$ and $k_i=k$, where $k =\underset{i \in [n]}{\mbox{min }}k_i$. Thus $\Big \lceil \frac{n}{1+\ceil{\frac{\Delta}{k}}}\Big \rceil \, \leq \, \gamma_{\overline{f}}$.\\
		Let $v_i$ be the vertex where $\ceil{\frac{d_i}{k_i}}$ is maximum. Then all but those $\ceil{\frac{d_i}{k_i}}$ vertices dominated by $v_i$ will form a $CDRD$-set for $G$. Hence $\gamma_{\overline{f}} \, \leq \, n-\underset{i \in [n]} {\mbox{max }}\bigg\lceil\frac{d_i}{k_i}\bigg\rceil$.
	\end{proof}
	The upperbound is attained for the star graph and the lower bound is attained for the complete graph $K_{3p}$ if $k_i\geq\lceil \frac{3p-1}{2}\rceil$ for all $v_i$ or the cycle $C_{3p}$ if $k_i=1$ for all $i \equiv 0\ (\textrm{mod}\ 3)$.
	\begin{theorem}
		For any isolate free graph $G$, $\gamma_{\overline{f}} \leq \beta'$, the edge covering number of $G$.
	\end{theorem}
	\begin{proof}
		Let $\{e_1,e_2,\hdots,e_{\beta'}\}$ be a maximum edge covering for the graph $G$. If $e_i$ has the end vertices $v_{i_1}$ and $v_{i_2}$, then the collection of all $v_{i_1}$'s forms a $CDRD$-set for the graph $G$. Thus $\gamma_{\overline{f}} \leq \beta'$.
	\end{proof}
	If all the vertices in the graph has the function value $k_i=d_i$, then the bound mentioned above will be attained by the $CDRD$ number $\gamma_{\overline{f}}$.
	\begin{coro}
		For any isolate free bipartite graph $G$, $\gamma_{\overline{f}} \leq \alpha$, the independence number of $G$.
	\end{coro}
	\subsubsection{Complete graph $K_n$}
	$CDRD$ number for a complete graph with some particular functions can be determined.
	\begin{theorem}
		For a complete graph $K_n$ if there are atmost $\big \lfloor \frac{2n}{3} \big \rfloor$ vertices with $k_i=d_i$, then $\gamma_{\overline{f}}(K_n) \leq \ceil{ \frac{n}{3} }$.
	\end{theorem}
	\begin{proof}
		Since $K_n$ is an $n-1$ regular graph of order $n$, the function value $k_i$ can vary from 1 to $n-1$. Unless $k_i=d_i$, $v_i$ can dominate atleast two vertices. Thus a set of atmost $\ceil{ \frac{n}{3} }$ vertices with $k_i<d_i$ will form a $CDRD$-set for $K_n$. Hence $\gamma_{\overline{f}}(K_n) \leq \ceil{ \frac{n}{3} }$.
	\end{proof}
	\noindent
	This is only a necessary condition for the $CDRD$ number to be bounded for $K_n$, but not sufficient since the presence of a vertex $v_i$ with $k_i=1$ and all other vertices with $k_i=d_i$ will give $\gamma_{\overline{f}}(K_n) \leq \ceil{ \frac{n}{3} }$.\\
	Similar result can be obtained for the $FDRD$ number as below.
	\begin{theorem}
		For a complete graph $K_n$ if there are atmost $\big \lfloor \frac{2n}{3} \big \rfloor$ vertices with $k_i>\ceil{\frac{n-1}{2} }$, then $\gamma_{\underline{f}}(K_n) \leq \ceil{ \frac{n}{3} }$.
	\end{theorem}
	\noindent
	Here also it is not a sufficient condition since a vertex with $k_i=1$ will give $\gamma_{\underline{f}}(K_n)=1$, without considering the $k_j$ values for other vertices.
	Similar result can be obtained for the $TDRD$ number as below.
	\begin{theorem}
		For a complete graph $K_n$ if there are atmost $\big \lfloor \frac{2n}{3} \big \rfloor$ vertices with $k_i=d_i$, then $\gamma_{f_t}(K_n) \leq \ceil{ \frac{n}{3} }$.
	\end{theorem}
	\noindent
	Here also it is not a sufficient condition since a vertex with $k_i=1$ will give $\gamma_{f_t}(K_n)=1$, without considering the $k_j$ values for other vertices.
	
	\subsubsection{Caterpillar graph}
	If $G$ is a caterpillar graph whose spine is the path $P_n=v_1v_2\hdots v_n$, where each vertex $v_i$ in the spine has degree $d_i$ and is attached to $l_i$ leaves, then $d_i= \bigg\{
	\begin{array}{ll}
		l_i+1 & \mbox{if } i=1 \mbox{ or } n\\
		l_i+2 & \mbox{if } i=2,3,\hdots,n-1
	\end{array}$, assuming that $l_i \geq 1$, for all $i=1,2,\hdots,n$.\\
	Let $f:V \rightarrow \mathbb{N}$ is defined as $f(v_i)=k_i$ and $f(v_{i_j})=1$, where $v_{i_j}$ is the leaf attached to the vertex $v_i,\,1\leq j \leq l_i$, and $1\leq i \leq n$.
	\begin{theorem}
		If $G$ is the caterpillar graph defined as above then $\gamma_{\overline{f}}(G)=\sum\limits_{i=1,n}^{}(l_i-\ceil{ \frac{l_i+1}{k_i}} )+\sum\limits_{i=2}^{n-1}(l_i-\ceil{ \frac{l_i+2}{k_i}})+n$, provided $l_i>0,k_i>1$ and if $l_i=1$ then $k_i=3$.
	\end{theorem}
	\begin{proof}
		Since $l_i>0$ for all the vertices in the spine $v_1v_2\hdots v_n$, each $v_i$ must be a member in every $CDRD$-set. Thus all the $n$ vertices in the spine are necessary in the $\gamma_{\overline{f}}$-set. In the spine, all but $v_1$ and $v_n$ can dominate $\ceil{ \frac{l_i+2}{k_i} }$ vertices and at the same time $v_1$ and $v_n$ can dominate $\ceil{ \frac{l_i+1}{k_i} }$ vertices. Let the number of pendant vertices adjacent to $v_i$ but not dominated by it, be $r_i$. Then $r_i= \bigg\{
		\begin{array}{ll}
			l_i-\ceil{ \frac{l_i+1}{k_i}} & \mbox{if } i=1 \mbox{ or } n\\
			l_i-\ceil{ \frac{l_i+2}{k_i}} & \mbox{if } i=2,3,\hdots,n-1
		\end{array}$.\\
		Any of the $r_i$ leaves adjacent to $v_i$ must be in every $CDRD$-set. So, $\gamma_{\overline{f}}(G)=\sum\limits_{i=1,n}^{}(l_i-\ceil{ \frac{l_i+1}{k_i}} )+\sum\limits_{i=2}^{n-1}(l_i-\ceil{ \frac{l_i+2}{k_i}})+n$.
	\end{proof}
	\begin{theorem}
		\label{star}
		For the star graph $K_{1,n}$, $\gamma_{\overline{f}} \geq \big \lfloor \frac{n}{2} \big \rfloor +1$ unless the central vertex has the function value 1.
	\end{theorem}
	\begin{proof}
		Let $v_1$ be the central vertex in the star graph $K_{1,n}$. If $f(v_1)=1$, then $\gamma_{\overline{f}}=\gamma=1$.\\
		If $f(v_1)=k_1$, where $1 < k_1 \leq n$. Since $\ceil{\frac{n}{k_1} } \leq \ceil{\frac{n}{2} }$, $v_1$ can dominate atmost $\ceil{\frac{n}{2} }$ remaining vertices of the graph. Thus any $CDRD$-set will contain atleast $n-\ceil{\frac{n}{2} }$ vertices of the graph other than $v_1$. Hence $\gamma_{\overline{f}} \geq \big \lfloor \frac{n}{2} \big \rfloor +1$ for any function $f:V \rightarrow \mathbb{N}$.
	\end{proof}
	\noindent
	Let $\mathfrak{C}_2$ denotes the collection of caterpillars whose leaves are attached only with the vertices $v_i$, where $i\equiv 2(\mbox{mod }3)$ on the spine $P_n=v_1v_2\hdots v_n$. Fig. \ref{fig:graph-2} is an example for a caterpillar in the class $\mathfrak{C}_2$.
	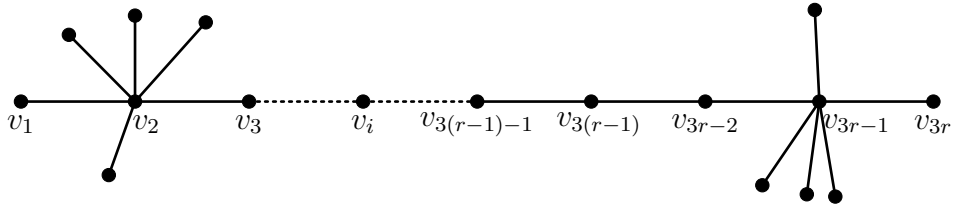
\begin{figure}[H]
		\centering
		\begin{tikzpicture}[line cap=round,line join=round,>=triangle 45,x=0.75cm,y=0.75cm]
			\clip(-5.,0.) rectangle (13.,4.);
			\draw [line width=1.pt] (-4.,2.)-- (-2.,2.);
			\draw [line width=1.pt] (6.,2.)-- (8.,2.);
			\draw [line width=1.pt] (8.,2.)-- (10.,2.);
			\draw [line width=1.pt] (-2.,3.52)-- (-2.,2.);
			\draw [line width=1.pt] (-2.,2.)-- (-0.76,3.4);
			\draw [line width=1.pt] (-3.16,3.18)-- (-2.,2.);
			\draw [line width=1.pt] (-2.,2.)-- (-2.46,0.7);
			\draw [line width=1.pt] (9.,0.52)-- (10.,2.);
			\draw [line width=1.pt] (10.,2.)-- (9.78,0.36);
			\draw [line width=1.pt] (10.28,0.32)-- (10.,2.);
			\draw [line width=1.pt] (10.,2.)-- (12.,2.);
			\draw [line width=1.pt] (10.,2.)-- (9.92,3.62);
			\draw [line width=1.pt] (-2.,2.)-- (0.,2.);
			\draw [line width=1.pt,dotted] (0.,2.)-- (2.,2.);
			\draw [line width=1.pt,dotted] (2.,2.)-- (4.,2.);
			\draw [line width=1.pt] (4.,2.)-- (6.,2.);
			\begin{scriptsize}
				\draw [fill=black] (-2.,2.) circle (2.5pt);
				\draw[color=black] (-1.8,1.6) node {{\normalsize $v_2$}};
				\draw [fill=black] (6.,2.) circle (2.5pt);
				\draw[color=black] (6.14,1.6) node {{\normalsize $v_{3(r-1)}$}};
				\draw [fill=black] (8.,2.) circle (2.5pt);
				\draw[color=black] (8,1.6) node {{\normalsize $v_{3r-2}$}};
				\draw [fill=black] (10.,2.) circle (2.5pt);
				\draw[color=black] (10.65,1.6) node {{\normalsize $v_{3r-1}$}};
				\draw [fill=black] (-4.,2.) circle (2.5pt);
				\draw[color=black] (-4,1.6) node {{\normalsize $v_1$}};
				\draw [fill=black] (-2.,3.52) circle (2.5pt);
				\draw [fill=black] (-0.76,3.4) circle (2.5pt);
				\draw [fill=black] (-3.16,3.18) circle (2.5pt);
				\draw [fill=black] (-2.46,0.7) circle (2.5pt);
				\draw [fill=black] (9.78,0.36) circle (2.5pt);
				\draw [fill=black] (10.28,0.32) circle (2.5pt);
				\draw [fill=black] (12.,2.) circle (2.5pt);
				\draw[color=black] (12,1.6) node {{\normalsize $v_{3r}$}};
				\draw [fill=black] (9.,0.52) circle (2.5pt);
				\draw [fill=black] (9.92,3.62) circle (2.5pt);
				\draw [fill=black] (0.,2.) circle (2.5pt);
				\draw[color=black] (0,1.6) node {{\normalsize $v_3$}};
				\draw [fill=black] (2.,2.) circle (2.5pt);
				\draw[color=black] (2.,1.6) node {{\normalsize $v_i$}};
				\draw [fill=black] (4.,2.) circle (2.5pt);
				\draw[color=black] (4,1.6) node {{\normalsize $v_{3(r-1)-1}$}};
			\end{scriptsize}
		\end{tikzpicture}
		\caption[$G_2$]{Graph $G \in \mathfrak{C}_2$}
		\label{fig:graph-2}
	\end{figure}
	
	\begin{theorem}
		If $G \in \mathfrak{C}_2$ with $d(v_i)=n_i$ and $f(v_i)=k_i$, then $\gamma_{\overline{f}}(G) \in \{k:k=\sum\limits_{\substack{i=1 \\ i\equiv 2 (\mbox{mod }3)}}^{n}n_i-\ceil{ \frac{n_i}{k_i} } +1\}$.
	\end{theorem}
	\begin{proof}
		Each $v_i$, where $i \equiv 2(\mbox{mod }3)$ together with the leaves attached to it forms a star graph $K_{1,n_i}$. So in Theorem \ref{star}, $n_i-\ceil{ \frac{n_i}{k_i} } +1$ vertices are required to dominate that star. Hence to dominate the entire graph we have to consider $CDRD$-set of each star centered at $v_i$, where $i \equiv 2(\mbox{mod }3)$.
	\end{proof}
	\subsubsection{Path and Cycle}
	\begin{theorem}
		If $P_n$ is a path of order $n$, $\ceil{ \frac{n}{3} } \leq \gamma_{\overline{f}}(P_n) \leq \ceil{ \frac{n}{2} }$.
	\end{theorem}
	\begin{proof}
		In a path $P_n\,=\,v_1v_2\hdots v_n$, degree of the vertex $v_i$ be $d_i$ and\\ $d_i= \bigg\{
		\begin{array}{ll}
			1 & \mbox{if } i=1 \mbox{ or } n\\
			2 & \mbox{if } i=2,3,\hdots,n-1
		\end{array}$.\\
		Thus $k_i= \bigg\{
		\begin{array}{ll}
			1 & \mbox{if } i=1 \mbox{ or } n\\
			1 \mbox{ or }2 & \mbox{if } i=2,3,\hdots,n-1
		\end{array}$. Hence $v_1$ and $v_n$ can dominate exactly one vertex and all other vertices will dominate 1 or 2 vertices with respect to the $k_i$ values 2 or 1 respectively. If each vertex $v_i$ has $k_i=d_i$, then each vertex can dominate exactly one other vertex and in such case, those vertices with odd index will form a $CDRD$-set. So, $\gamma_{\overline{f}}(P_n) \leq \ceil{ \frac{n}{2} }$. Also $\gamma(P_n)=\ceil{ \frac{n}{3} }$ and $\gamma(G) \leq \gamma_{\overline{f}}(G)$ for any graph $G$. Thus $\ceil{ \frac{n}{3} } \leq \gamma_{\overline{f}}(P_n) \leq \ceil{ \frac{n}{2} }$.
	\end{proof}
	\begin{theorem}
		If $C_n$ is a cycle of order $n$, $\ceil{ \frac{n}{3} } \leq \gamma_{\overline{f}}(C_n) \leq \ceil{ \frac{n}{2} }$.
	\end{theorem}
	\begin{definition}[Restricted radius]
		Let $G=(V,E)$ be any graph and $S \subseteq V$. We can define the radius of $G$ restricted to $S$ as $\mbox{rad}_S(G)=\underset{u,v \in S}{min}\{d_G(u,v)\}$.
	\end{definition}
	\noindent
	Clearly, if $S_1 \subseteq S_2$, then $\mbox{rad}_{S_1}(G) \geq \mbox{rad}_{S_2}(G)$.
	\begin{theorem}
		For a cycle $C_n$ of order n with the given function f, if $\mbox{rad}_{S_1}(G)\geq3$ where $S_1$ is the collection of vertices mapping to 1 by the function f, then $\gamma_{\overline{f}}(C_n)\leq n-2|S_1|$.
	\end{theorem}
	\begin{proof}
		For any $v_i \in V(C_n)$, $d_i=2$ and so $k_i=1 \mbox{ or }2$. Let $S_1=\{v_i \in V(C_n):f(v_i)=1\}$. If $\mbox{rad}_{S_1}(G)\geq3$, there exist vertices in $V$ with both the function values 1 and 2. All the vertices in $S_1$ can dominate both of its neighbours and thus, the remaining vertices which are not being dominated are only those with $k_i=2$, that is, $n-3|S_1|$ in number. All these vertices together with those in $S_1$ will form a $CDRD$-set. Hence $\gamma_{\overline{f}}(C_n)\leq n-2|S_1|$.\\
		If $\mbox{rad}_{S_1}(G)<3$, then can form a new set $S$ from $S_1$ by deleting necessary vertices so that $\mbox{rad}_{S}(G)\geq3$. Then $\gamma_{\overline{f}}(C_n)\leq n-2|S|$.
	\end{proof}
	
	\section{Conclusion}
	
		In this paper, some new generalized forms of domination were introduced by restricting the number of vertices a vertex can dominate. The newly defined variations of dominations are Ceil Degree Restricted Domination, Floor Degree Restricted Domination and Translate Degree Restricted Domination. As the name indicates, in each type of domination a vertex dominates a particular number of vertices as the given function indicates. Some bounds for the Degree Restricted Domination number have been discussed in this paper and also Degree Restricted Domination is studied for some particular classes of graphs like the complete graph $K_n$, paths, cycles and the caterpillars.


\begin{thebibliography}{}
		\bibitem{haynes1998fundamentals}
		Teresa W Haynes, Stephen T Hedetniemi, and Peter J Slater. Fundamentals of domination in graphs. Marcel Dekker Inc., New York, 1998.
		
		\bibitem{haynes1998domination}
		Teresa W Haynes, Stephen T Hedetniemi, and Peter J Slater. Domination in graphs: Volume 2: advanced topics. Marcel Dekker Inc., New York, 1998.
		
		\bibitem{kamath20162drd}
		S S Kamath, A Senthil Thilak, and M Rashmi. 2-part degree restricted domination in graphs. In IWCAAM'16, pp 211-221. SHANLAX, 2016.
		
		\bibitem{kamath2019relation}
		S S Kamath, A Senthil Thilak, and M Rashmi. Relation between k-DRD and dominating set. In Applied Mathematics and Scientific Computing, pp 563-572. Springer, 2019.
		
		\bibitem{west2001introduction}
		West, D.~B,  Introduction to graph theory, Pearson Education, India, 2nd edition (2001).
	\end{thebibliography}
\end{document}